\documentclass[10 pt, onecolumn, oneside, a4paper, reqno]{amsart}
\usepackage{amssymb}
\usepackage{booktabs}
\usepackage{mathrsfs}
\usepackage{mathtools}
\usepackage{amsmath}
\usepackage{tikz}
\usepackage{pifont}
\usepackage{xcolor}
\usepackage[all]{xy}
\usepackage{type1cm}
\usepackage{hyperref}
\usepackage{color}
\hypersetup{colorlinks=true,
     breaklinks=true,
     linkcolor=blue,
      pageanchor=true}

\newtheorem{thm}{Theorem}[section]

\newtheorem{cor}[thm]{Corollary}
\newtheorem{lem}[thm]{Lemma}
\newtheorem{prop}[thm]{Proposition}

\theoremstyle{definition}
\newtheorem{defn}[thm]{Definition}

\newcommand{\cpx}[1]{#1^{\bullet}}

 \newcommand{\A}{\mathcal A}

\newcommand{\C}{\mathcal C}
\newcommand{\D}{\mathcal D}

\newcommand{\GF}{\mathsf {GF}}
\newcommand{\F}{\mathsf F}
\newcommand{\G}{\mathsf G}
\newcommand{\GP}{\mathcal {GP}}
\renewcommand{\H}{\mathcal H}
\newcommand{\I}{\mathcal I}

\newcommand{\MF}{\mathsf{MF}}

\renewcommand{\S}{\mathcal S}

\newcommand{\T}{\mathcal T}

\newcommand{\X}{\mathcal X}
\newcommand{\Y}{\mathcal Y}

\DeclareMathOperator{\Cok}{\mathsf{Cok}}

\DeclareMathOperator*{\Gproj}{\!-\mathsf{Gproj}}
\DeclareMathOperator*{\GProj}{\!-\mathsf{GProj}}
\DeclareMathOperator{\Hom}{\mathsf{Hom}}
 \DeclareMathOperator*{\id}{\mathsf{id}}
  
 \DeclareMathOperator*{\inc}{\mathsf{inc}}

 \DeclareMathOperator*{\Ker}{\mathsf{Ker}}
\DeclareMathOperator*{\Mod}{\!-\mathsf{Mod}}
\DeclareMathOperator*{\Mor}{\mathsf{Mor}}

\DeclareMathOperator*{\pr}{\mathsf{pr}}

\DeclareMathOperator*{\Proj}{\!-\mathsf{Proj}}

\title[]{$N$-fold module factorizations: triangle equivalences and recollements}

\dedicatory{Dedicated to Professor Bin Zhu on the occasion of his 60th birthday}

\author[Y. Sun & Y. Zhang]{Yongliang Sun and Yaohua Zhang*}

\address{\normalfont{Yongliang Sun \\School of Mathematics and Physics, Yancheng Institute of Technology, 224003 Jiangsu, People's Republic of China}}
\email{syl13536@126.com}

\address{\normalfont{Yaohua Zhang \\ Hubei Key Laboratory of Applied Mathematics, Faculty of Mathematics and Statistics, Hubei University, 430062 Wuhan, People's Republic of China}}
\email{yhzhang@hubu.edu.cn}

\thanks{* Corresponding author.}



\begin{document}
\begin{abstract}
As an extension of Eisenbud's matrix factorization into the non-commutative realm, X.W. Chen introduced the concept of module { factorizations} over an arbitrary ring. A theorem of Chen establishes a { triangle equivalences} between the stable category of module factorizations with Gorenstein projective components and the stable category of Gorenstein projective modules over a quotient ring. In this paper, we introduce $n$-fold module { factorizations}, which { generalize} both the commutative $n$-fold matrix factorizations and the non-commutative module factorizations. To adapt { triangle equivalences} in module factorizations to $n$-fold module factorizations, we identify suitable subcategories of module factorizations and rings for the $n$-analogue. We further provide the $n$-analogue of Chen's theorem on { triangle equivalences}. Additionally, we {  study} recollements involving the stable categories of higher-fold module factorizations, revealing intriguing recollements within the stable categories of Gorenstein modules of specific matrix subrings.
\end{abstract}
\maketitle
\section{Introduction}
The concept of matrix factorizations was originally introduced by Eisenbud \cite{E80} to study periodic free resolutions over hypersurface singularities. Since then, matrix factorizations have found applications in a wide range of mathematical fields, including tensor triangular geometry \cite{Y16}, commutative algebra \cite{E80}, knot theory \cite{KR08}, string theory \cite{C09}, Hodge theory \cite{BFK14}, singularity categories \cite{B21, O04, PV11}, and quiver representations \cite{KST07}.
  {This concept has been extended to the noncommutative setting by Cassidy, Conner, Kirkman, and Moore \cite{CCKM}, who introduced the notion of {twisted matrix factorizations}. More generally, Bergh and Erdmann \cite{BE19} further generalized this idea by replacing free modules with arbitrary modules, which leads to a definition that coincides with the notion of {module factorizations} introduced by Chen \cite{Chen24}. The noncommutative version also admits a variety of applications; see for instance \cite{CKMW, MU}.}

Motivated by these developments, this paper introduces the notion of $n$-fold module factorizations, which unify and generalize two important classes of objects: the (2-fold) module factorizations studied by Bergh and Erdmann \cite{BE19}, and independently by Chen \cite{Chen24}, as well as the $n$-fold matrix factorizations investigated by Hopkins and Tribone \cite{H21, T21}. The main objective of this work is to establish a connection between $n$-fold module factorizations with Gorenstein components and Gorenstein projective modules over a certain class of upper triangular matrix rings. This investigation naturally leads to the study of triangulated equivalences among specific stable categories associated with these constructions. In doing so, we extend the existing theory in two directions: first, by generalizing the commutative framework of $n$-fold matrix factorizations, and second, by extending the noncommutative setting from 2-fold to $n$-fold module factorizations.

Let $A$ be a left noetherian ring, and let $\omega$ be a regular and normal element in $A$. The element $\omega$ induces a unique ring automorphism $\sigma: A\to A$ such that $\omega a=\sigma(a)\omega$ for $a\in A$, along with the quotient ring $A/(\omega)$, denoted as $\bar{A}$. { Given an $A$-module $M$, }there is a canonical morphism $\omega_M:M\to {^\sigma}(M), m\mapsto {^\sigma}(\omega m)$, where ${^\sigma}(M)$ has the same elements as $M$ but with the twisted action given by $a {{^\sigma}(m)}={^\sigma}(\sigma(a)m)$ for all $a\in A$ and $m\in M$.

Let $n$ be a positive integer. An {\em $n$-fold module factorization} $(X^i,d_X^i)$ of $\omega$ is defined as a sequence
$$\xymatrix{
X^0\ar[r]^{d_X^{0}}&X^1\ar[r]^{d_X^{1}}&\cdots\ar[r]^{d_X^{n-2}}&X^{n-1}\ar[r]^{d_X^{n-1}}&{^\sigma(X^0)}
}$$
of $n$ morphisms in $A\Mod$, where each possible $n$-fold composition equals $\omega$ (see Subsection~\ref{subsec:n mod fac}).
Denote by $\F_n(A;\omega)$ the category of $n$-fold module factorizations of $\omega$. In particular, we identify $\F_1(A;\omega)$ as $A\Mod$. There are subcategories
\begin{align*}
   \GF_n(A;\omega):=&\{X\in\F_n(A;\omega)\mid X^i~\text{is Gorenstein-projective for}~0\leq i\leq n-1\} \\
   \G^0\F_n(A;\omega):=&\{X\in\GF_n(A;\omega)\mid X^{n-1}~\text{is projective}\} \\
   \MF_n(A;\omega):=&\{X\in\F_n(A;\omega)\mid X^{i}~\text{is finitely generated projective for}~0\leq i\leq n-1\}
\end{align*}

Define the following two matrix rings
$$
\quad\quad\Gamma_n=\begin{pmatrix}
       A& A&\cdots & A\\
       A\omega & A &\cdots &A\\
       \vdots &\vdots &\ddots &\vdots\\
       A\omega & A\omega &\cdots &A
   \end{pmatrix}_{n\times n}\quad\quad
\Lambda_{n-1}=\begin{pmatrix}
    \bar{A}&\bar{A}&\cdots &\bar{A}\\
    0 &\bar{A}&\cdots&\bar{A}\\
    \vdots&\vdots&\ddots&\vdots\\
    0&0&\cdots&\bar{A}
\end{pmatrix}_{(n-1)\times (n-1)}
$$

In Section~\ref{sec:module factorizations}, we realize the category $\F_n(A;\omega)$ of $n$-fold module factorizations of $\omega$ as the module category $\Gamma_n\Mod$. Consequently, the subcategory $\GP(\F_n(A;\omega))$, which consists of Gorenstein projective $n$-fold module factorizations, is identified as a Frobenius category. Specifically, we prove the following theorem, which generalizes \cite[Proposition 3.4, Proposition 3.11]{Chen24}. 

\begin{thm}{\textnormal{(Proposition~\ref{prop:module cat}, Proposition~\ref{prop:projs} and Corollary~\ref{cor:Gorenstein stable})}}
    The category $\F_n(A;\omega)$ is equivalent to $\Gamma_n\Mod$. The subcategory $\GF_n(A;\omega)$ coincides with $\GP(\F_n(A;\omega))$. Moreover, there is a triangle equivalence $\underline{\GF}_n(A;\omega)\simeq \Gamma_n\underline{\GProj}$.
\end{thm}

In Section~\ref{sec:cokernel}, based on the observation that an $n$-fold module factorization can be decomposed into a composition of $(n-1)$ morphisms of 2-fold module factorizations, we define the zeroth cokernel functor $\Cok^0$ that maps the category $\F_n(A;\omega)$ of $n$-fold module factorizations to the module category $\Lambda_{n-1}\Mod$. We establish the following triangle equivalences, which can be interpreted as a noncommutative $n$-fold version of Eisenbud's theorem (\cite{E80,H21}). In contrast to the the classical case $n=2$, we characterize the category of Gorenstein projective modules over an upper-triangular matrix ring using monomorphism categories, as developed by X.-H. Luo and P. Zhang \cite{LZ13}. When $n=2$, the theorem below aligns with \cite[Theorem 5.6, Theorem 7.2]{Chen24}.

\begin{thm}{\textnormal{(Theorem~\ref{thm:cok equiv 1} and Theorem~\ref{thm:matrix factor})}}
    The zeroth cokernel functor $\Cok^0$ induces triangle equivalences   $$\Cok^0:\underline{\G^0\F}_n(A;\omega)\stackrel{\sim}\longrightarrow \Lambda_{n-1}\underline{\GProj}$$
    and
    $$\underline{\MF}_n(A;\omega)\stackrel{\sim}\longrightarrow \Lambda_{n-1}\underline{\Gproj}^{<+\infty}.$$
\end{thm}

In Section~\ref{sec:recollement}, we explore recollements of the triangulated stable categories discussed in the preceding sections. For precise definitions of the functors mentioned below, please refer to Section~\ref{sec:preliminaries}.

\begin{thm}{\textnormal{(Theorem~\ref{thm:recollement})}}
     There is the following recollement for $1\leq k<n$
$$\xymatrix{
{\underline{\G^0\F}_{n-k+1}(A;\omega)}\ar[rr]^{\inc}
&
&
{\underline{\GF}_{n}(A;\omega)}\ar[rr]^{{\pr}_{k+1}^0\cdots {\pr}_n^0}\ar@/^1.2pc/[ll]^{\inc_{\rho}}\ar@/_1.3pc/[ll]_{\inc_{\Lambda_{n-1}}}
&
&{\underline{\GF}_{k}(A;\omega)}\ar^-{S^n\theta_n^0\cdots\theta_n^0 S^{-(k-1)}}@/^1.2pc/[ll]\ar_-{\theta_n^0\cdots\theta_{k+1}^0}@/_1.3pc/[ll]
}$$  
The recollement restricts recollements
$$\xymatrix{
{\underline{\G^0\F}_{n-k+1}(A;\omega)}\ar[rr]^{\inc}
&
&
{\underline{\G^0\F}_{n}(A;\omega)}\ar[rr]^{{\pr}_{k+1}^0\cdots {\pr}_n^0}\ar@/^1.2pc/[ll]^{\inc_{\rho}}\ar@/_1.3pc/[ll]_{\inc_{\Lambda_{n-1}}}
&
&{\underline{\G^0\F}_{k}(A;\omega)}\ar^-{S^n\theta_n^0\cdots\theta_n^0 S^{-(k-1)}}@/^1.2pc/[ll]\ar_-{\theta_n^0\cdots\theta_{k+1}^0}@/_1.3pc/[ll]
}$$
and 
$$\xymatrix{
{\underline{\MF}_{n-k+1}(A;\omega)}\ar[rr]^{\inc}
&
&
{\underline{\MF}_{n}(A;\omega)}\ar[rr]^{{\pr}_{k+1}^0\cdots {\pr}_n^0}\ar@/^1.2pc/[ll]^{\inc_{\rho}}\ar@/_1.3pc/[ll]_{\inc_{\Lambda_{n-1}}}
&
&{\underline{\MF}_{k}(A;\omega)}.\ar^-{S^n\theta_n^0\cdots\theta_n^0 S^{-(k-1)}}@/^1.2pc/[ll]\ar_-{\theta_n^0\cdots\theta_{k+1}^0}@/_1.3pc/[ll]
}$$
\end{thm}

By setting $n=2$ and $k=1$, the first recollement aligns exactly with the one presented in \cite[Proposition 6.3]{Chen24}. As a straightforward consequence of the preceding theorem, we obtain recollements for the stable categories of Gorenstein modules over the matrix subrings $\Lambda_m$ and $\Gamma_n$.

\begin{cor}{\textnormal{(Corollary~\ref{cor:recollement})}}
    There is the following recollement for $1\leq k<n$
    
$$\xymatrix{
\Lambda_{n-k}\underline{\GProj}\ar[rr]
&
&
\Gamma_n\underline{\GProj}\ar[rr]_{\quad}\ar@/^1.2pc/[ll]\ar@/_1.3pc/[ll]_{\quad}
&
&\Gamma_k\underline{\GProj}\ar@/^1.2pc/[ll]_{\quad}\ar@/_1.3pc/[ll]_{\quad}^{\quad}
}$$
which can restrict to recollements

$$\xymatrix{
\Lambda_{n-k}\underline{\GProj}\ar[rr]_{\quad}^{\quad}
&
&
\Lambda_{n-1}\underline{\GProj}\ar[rr]\ar@/^1.2pc/[ll]_{\quad}^{\quad}\ar@/_1.3pc/[ll]
&
&\Lambda_{k-1}\underline{\GProj}\ar@/^1.2pc/[ll]\ar@/_1.3pc/[ll]_{\quad}^{\quad}
}$$
and
$$\xymatrix{
\Lambda_{n-k}\underline{\Gproj}^{<+\infty}\ar[rr]_{\quad}^{\quad}
&
&
\Lambda_{n-1}\underline{\Gproj}^{<+\infty}\ar[rr]\ar@/^1.2pc/[ll]\ar@/_1.3pc/[ll]_{\quad}^{\quad}
&
&\Lambda_{k-1}\underline{\Gproj}^{<+\infty}.\ar@/^1.2pc/[ll]\ar@/_1.3pc/[ll]_{\quad}^{\quad}
}$$
\end{cor}

\vspace{0.3cm}
\section{Preliminaries}\label{sec:preliminaries}
In this section, we review Gorenstein projective modules over a ring and introduce the concept of $n$-fold module factorizations. To facilitate our exploration, we establish canonical functors pertinent to the categories of module factorizations.

\subsection{Setting}
Let $A$ be a left noetherian ring. An element $\omega\in A$ is {\em regular} if $a\omega=0$ or $\omega a=0$ implies $a=0$. An element $\omega\in A$ is {\em normal} if $A\omega=\omega A$. A regular and normal element $\omega$ induces a unique ring automorphism $\sigma: A\to A$ such that $\omega a=\sigma(a)\omega$ for $a\in A$, in this case, $\sigma(\omega)=\omega$. Denote $\bar{A}=A/(\omega)$, the module category  $\bar{A}\Mod$ can be regarded as a full subcategory of $A\Mod$ consisting of all modules vanished by $\omega$. For an $A$-module $X$, define its {\em twisted module} ${^\sigma}(X)$ by $a{^\sigma}(x)={^\sigma}(\sigma(a)x)$. There is a canonical morphism of $A$-modules
\begin{align*}
    \omega_X:X&\longrightarrow {^\sigma}(X) \\
             x&\longmapsto {^\sigma}(\omega x).
\end{align*}

Let 
$$\xymatrix{
X^0\ar[r]^{d_X^{0}}&X^1\ar[r]^{d_X^{1}}&\cdots\ar[r]^{d_X^{n-2}}&X^{n-1}\ar[r]^{d_X^{n-1}}&{^\sigma(X^0)}
}$$
be a sequence of $A$-modules, we denote its composition
$d_X^j\cdots d_X^i$ for $i\leq j$ by $d_X^{i,j}$ simply.
 
\subsection{Gorenstein projective modules}
Let $\cpx{P}=(P^n,d_{P}^n)$ be a complex of projective $A$-modules. The complex $\cpx{P}$ is called {\em totally acyclic} if it is exact and if, for each projective $A$-module $Q$, the Hom complex $\Hom_{A}(\cpx{P}, Q)$ is acyclic. An $A$-module $G$ is {\em Gorenstein projective} if there exists a totally acyclic projective complex $\cpx{P}$ such that $\Ker(d_{P}^0)$ is isomorphic to $G$. We denote by $A\GProj (A\Gproj)$ the full subcategory of (finitely generated) Gorenstein projective $A$-modules.

The lemma below plays a crucial role in proving the key lemma (Lemma~\ref{full and dense}).

\begin{lem}{\textnormal{(\cite[Theorem 2.13]{Chen24})}}\label{GP under syzygy}
 Assume that $0\to L\to G\to N\to 0$ is a short exact sequence of $A$-module with $G\in A\GProj$ and $\omega_{N}=0$. Then $N$ belongs to $\bar{A}\GProj$ if and only if $L$ belongs to $A\GProj$.
\end{lem}

\subsection{$N$-fold module factorizations}\label{subsec:n mod fac}
Let $n$ be a positive integer, an {\em $n$-fold module factorization} $(X^i,d_X^i)$ of $\omega$ is a sequence
$$\xymatrix{
X^0\ar[r]^{d_X^{0}}&X^1\ar[r]^{d_X^{1}}&\cdots\ar[r]^{d_X^{n-2}}&X^{n-1}\ar[r]^{d_X^{n-1}}&{^\sigma(X^0)}
}$$
of $n$ morphisms in $A\Mod$, in which each $n$-fold composition equals $\omega$:
\begin{align*}
    d_X^{0, n-1}&=\omega_{X^0}\\
    {^\sigma(d_X^0)}d_X^{1, n-1} &=\omega_{X^1}\\
       &\vdots\\
      {^\sigma(d_X^{0,n-2})}\cdots d_X^{n-1}&=\omega_{X^{n-1}}.
\end{align*}
A {\em morphism} $f:(X^i, d_X^i)\to (Y^i, d_Y^i)$ of two $n$-fold module factorizations of $\omega$ is a sequnence $(f^0, f^1,\cdots, f^{n-1})$ of morphisms in $A\Mod$ such that the following diagram commutes
$$\xymatrix
{X^0\ar[r]^{d_X^0}\ar[d]^{f^0}&X^1\ar[r]^{d_X^1}\ar[d]^{f^1}&\cdots\ar[r]^{d_X^{n-2}}&X^{n-1}\ar[r]^{d_X^{n-1}}\ar[d]^{f^{n-1}}&{^\sigma(X^0)}\ar[d]^{^\sigma(f^0)}
\\
Y^0\ar[r]^{d_Y^0}&Y^1\ar[r]^{d_Y^1}&\cdots\ar[r]^{d_Y^{n-2}}&Y^{n-1}\ar[r]^{d_Y^{n-1}}&{^\sigma(Y^0)}
}$$

Denote by $\F_n(A;\omega)$ the category of $n$-fold module factorizations of $\omega$. In particular, we identify $\F_1(A;\omega)$ as $A\Mod$. 

\subsection{Shift functors and trivial module factorizations}
 For an $n$-fold module factorization $X=(X^i, d_X^i)$, we define its {\em shift module factorization} as 
$$\xymatrix{
S(X): X^1\ar[r]^{\quad d_X^{1}}&X^2\ar[r]^{d_X^{2}}&\cdots\ar[r]&X^{n-1}\ar[r]^{d_X^{n-1}}&{^\sigma(X^0)}\ar[r]^{{^\sigma(X^0)}}&{^\sigma(X^1)}.
}$$
 The inverse is denoted by $S^{-1}$.  In contrast to \cite[Definition3.1]{Chen24}, we choose not to include a minus sign for each morphism in the shift module factorization, which simplifies matters for an arbitrary integer $n$. In fact, using Chen's definition would introduce significant complications, particularly where $n$ is odd. For an $A$-module $M$, we further define its associated {\em trivial $n$-fold module factorizations}

{ \begin{align*}
     \theta^0(M):&~ 
{\xymatrix{M\ar[r]^{\id}&M\ar[r]^{\id}&\cdots\ar[r]^{\id}&M\ar[r]^{\omega\quad}&{^\sigma(M)}}}
\\
\theta^1(M):&~ {\xymatrix{
{^{\sigma^{-1}}(M)}\ar[r]^{\quad\omega}&M\ar[r]^{\id}&\cdots\ar[r]^{\id}&M\ar[r]^{\id}&M
}}\\
&\quad\quad\quad\vdots\\
\theta^i(M):&~ 
{\xymatrix{
{^{\sigma^{-1}}(M)}\ar[r]^{\quad\id}&\cdots\ar[r]^{\id\quad}&{^{\sigma^{-1}}(M)}\ar[r]^{\quad\omega}&M\ar[r]^{\id\quad}&\cdots\ar[r]^{\id}&M
}}\\
&\quad\quad\quad\vdots\\
\theta^{n-1}(M):&~ 
{\xymatrix{
{^{\sigma^{-1}}(M)}\ar[r]^{\id}&{^{\sigma^{-1}}(M)}\ar[r]^{\quad\id}&\cdots\ar[r]^{\id\quad}&{^{\sigma^{-1}}(M)}\ar[r]^{\quad\omega}&M\ar[r]^{\id}&M
}}
\end{align*}}

This gives rise to fully-faithful functors
 $$\theta^i:A\Mod \longrightarrow \F_n(A;\omega), i=0,1,\cdots, n-1.$$

Also, we have projective functors
$${\pr}^i: \F_n(A;\omega)\longrightarrow A\Mod, X\mapsto X^i, i=0,1,\cdots, n-1.$$

The following lemma is trivial, we omit its proof.

\begin{lem}
The following statements hold.
   \begin{enumerate}
       \item $S\theta^{i+1}=\theta^i$ for $i=0,\cdots,n-2$;
       \item ${\pr}^j\theta^i=\id_{A\Mod}$ for $0\leq i\leq j\leq n-1$.
   \end{enumerate} 
\end{lem}
\subsection{Face functors and degeneracy functors}

We define canonical functors between $\F_n(A;\omega)$ and $\F_{n+1}(A;\omega)$. The {\em face functors} are defined as
$$\theta_n^i:\F_n(A;\omega)\longrightarrow \F_{n+1}(A;\omega),i=0,1,\cdots,n$$
which takes $X$ in $\F_n(A;\omega)$ to
$$\xymatrix{
X^0\ar[r]^{d_X^{0}}&\cdots\ar[r]^{d_X^{i-1}}& X^i\ar[r]^{\id_{X^i}}&X^i\ar[r]^{d_X^{i}}&X^{i+1}\ar[r]^{d_X^{i+1}}&\cdots\ar[r]^{d_X^{n-2}}&X^{n-1}\ar[r]^{d_X^{n-1}}&{^\sigma(X^0)}
}$$
and the {\em degeneracy functors} are defined as
$${\pr}_{n+1}^{i}:\F_{n+1}(A;\omega)\longrightarrow \F_{n}(A;\omega),i=0,1,\cdots,n$$
which takes $Y$ in $\F_{n+1}(A;\omega)$ to
$$\xymatrix{
Y^0\ar[r]^{d_Y^{0}}&\cdots\ar[r]^{d_Y^{i-1}}&Y^i\ar[r]^{d_Y^{i,i+1}}&Y^{i+2}\ar[r]^{d_Y^{i+2}}&\cdots\ar[r]^{d_Y^{n-1}}&Y^{n}\ar[r]^{d_Y^{n}}&{^\sigma(Y^0)}
}$$

Each face functor is fully faithful, while each degeneracy functor is dense. The following relationships among the canonical functors are standard, and we will omit their proofs.

\begin{lem}\label{pr}
\begin{enumerate}
    \item $\theta^0=\theta_{n-1}^0\cdots \theta_1^0$;
    \item ${\pr}^i={\pr}_2^0\cdots {\pr}_n^0 S^i={\pr}^0S^i$ for $i=0,\cdots,n-1$.
\end{enumerate}
\end{lem}

\begin{lem}\label{i to i+1}
For $i=0,1,\cdots,n$, 
\begin{enumerate}
    \item ${\pr}_{n+1}^i\theta_n^i=\id_{\F_n(A;\omega)}$;
    \item $S\theta_n^{i+1}=\theta_n^{i}S$;
    \item ${\pr}_{n+1}^i S=S {\pr}_{n+1}^{i+1}$.
\end{enumerate}
\end{lem}

\begin{lem}\label{lem:adjoint pair}
The following statements hold.
\begin{enumerate}
    \item The pair $(\theta_n^0, {\pr}_{n+1}^0): \F_n(A;\omega)\to \F_{n+1}(A;\omega)$ is an adjoint pair.
    \item The pair $({\pr}_{n+1}^{n-1}, S\theta_n^0):\F_{n+1}(A;\omega)\to \F_{n}(A;\omega)$ is an adjoint pair.
\end{enumerate}
\end{lem}
\begin{proof}
(1) Let $X\in \F_n(A;\omega)$ and $Y\in \F_{n+1}(A;\omega)$, the canonical map
\begin{align*}
    \Hom_{\F_{n+1}(A;\omega)}(\theta_n^0(X),Y)&\longrightarrow \Hom_{\F_{n}(A;\omega)}(X, {\pr}_{n+1}^0(Y))\\
    (f^0,{f^0}', f^1,\cdots, f^{n-1})&\longmapsto (f^0, f^1,\cdots, f^{n-1})
\end{align*}
admits an inverse which takes a morphism $(f^0, f^1,\cdots, f^{n-1})$ in $\Hom_{\F_{n}(A;\omega)}(X, {\pr}_{n+1}^0(Y))$ to a morphism $(f^0,{f^0}', f^1,\cdots, f^{n-1})$ in $\Hom_{\F_{n+1}(A;\omega)}(\theta_n^0(X),Y)$
$$\xymatrix{
X^0\ar[r]^{\id}\ar[d]_{f^0}&X^0\ar[r]^{d_X^0}\ar[d]|{{f^0}'=d_Y^0f^0}&X^1\ar[r]\ar[d]_{f^2}&\cdots\ar[r]&X^{n-1}\ar[r]^{d_X^{n-1}}\ar[d]_{f^n}&{^{\sigma}(X^0)}\ar[d]^{^{\sigma}(f^0)}
\\
Y^0\ar[r]^{d^0_Y}&Y^1\ar[r]^{d_Y^1}&Y^2\ar[r]&\cdots\ar[r]&Y^n\ar[r]^{d^{n}_Y}&{^{\sigma}(Y^0)}
}$$

(2)  Let $X\in \F_n(A;\omega)$ and $Y\in \F_{n+1}(A;\omega)$, the canonical map
\begin{align*}
   \Hom_{\F_n(A;\omega)}({\pr}^{n-1}_{n+1}(Y), X)&\longrightarrow \Hom_{\F_{n+1}(A;\omega)}(Y, S\theta_n^0(X)) \\
   (f^0, f^1,\cdots, f^{n-1})&\longmapsto (f^0, f^1,\cdots, f^{n-1}, {^\sigma{f^0}d_Y^n})
\end{align*}
admits an inverse 
$$(f^0, f^1,\cdots, f^{n-1}, f^n)\longmapsto (f^0, f^1,\cdots, f^{n-1}).$$
\end{proof}

\begin{cor}
The following hold.
\begin{enumerate}
    \item Each pair $(\theta_n^i, {\pr}_{n+1}^i)$ is an adjoint pair, for $i=0, 1,\cdots, n-1$
    \item The pair $({\pr}_{n+1}^0, S^n\theta_n^0S^{-(n-1)})$ is an adjoint pair.
\end{enumerate}
\end{cor}

\begin{proof}
It follows directly from Lemma~\ref{lem:adjoint pair} and Lemma~\ref{i to i+1}. To prove (2), one should note that ${\pr}_{n+1}^{n-1}=S^{-(n-1)}{\pr}_{n+1}^0S^{n-1}$ (from Lemma \ref{i to i+1}(3)).
\end{proof}

\subsection{P-null-homotopical morphisms and homotopy category of module factorizations}
\begin{def}\label{def:p-null-homo}
    A morphism $(f^0, f^1,\cdots, f^{n-1}):X\to Y$ of $n$-fold module factorizations is {\em p-null-homotopical} (compare with \cite[Definition 4.1]{Chen24} and \cite[Defenition 2.2.1]{H21}), if there exist $h^0: X^0\to {^{\sigma^{-1}}(Y^{1})}, h^1:X^1\to {^{\sigma^{-1}}(Y^{2})}, \cdots,h^{n-1}:X^{n-1}\to Y^0$ of $A$-modules such that all $h^i$ factor through projective $A$-modules and satisfy
    \begin{align*}
        f^0=&{^{\sigma^{-1}}}(d_Y^{1,n-1})h^0+{^{\sigma^{-1}}}(d_Y^{2, n-1})h^1d_X^0+\cdots+h^{n-1}d_X^{0,n-2}\\
        &\quad\quad\quad\quad\vdots\\
        f^i=&d_Y^{0,i-1}{^{\sigma^{-1}}(d_Y^{i+1,n-1})}h^i+\cdots+d_Y^{0,i-1} h^{n-1} d_X^{i,n-2}\\
            &+d_Y^{1,i-1}{^{\sigma}(h^0)}d_X^{i,n-1}+\cdots + {^{\sigma}(h^{i-1}d_X^{0,i-2})}d_X^{i,n-1} \\
        &\quad\quad\quad\quad\vdots\\
        f^{n-1}=&d_Y^{0,n-2} h^{n-1}+d_Y^{1,n-2}{^{\sigma}(h^0)}d_X^{n-1}+\cdots + {^{\sigma}(h^{n-2}d_X^{0,n-3})}d_X^{n-1}
    \end{align*}
\end{def}

 Two morphisms $f$ and $g$ are said to be {\em homotopic} if $f-g$ is p-null homotopical, denoted by $f\sim g$. Homotopy is an equivalence relation. 
 We define the {\em homotopy category}  of $\F_n(A;\omega)$ by
$$\H(\F_n(A;\omega)):=\F_n(A;\omega)/\sim.$$

\begin{lem}\label{lem:p-null-homo}
A morphism $(f^0, f^1,\cdots, f^{n-1}): X\to Y$ between two n-fold module factorizations is p-null-homotopical if and only if it factors through $\theta^0(P^0)\oplus\theta^1(P^1)\oplus\cdots\oplus\theta^{n-1}(P^{n-1})$ for some projective $A$-modules $P^0,P^1,\cdots,P^{n-1}$.
\end{lem}
\begin{proof}
For the ``only if" part, we assume there are the following decompositions
\begin{align*}
    h^i:& 
   {\xymatrix{
   X^i\ar[r]^{\mu^{i}\quad}&{^{\sigma^{-1}}(P^{i+1})}\ar[r]^{^{\sigma^{-1}}(\nu^{i})\quad\quad\quad\quad}&{^{\sigma^{-1}}(Y^{i+1})},~i=0,1,\cdots,n-2, 
   }} 
    \\
    h^{n-1}: &
 {\xymatrix{
X^{n-1}\ar[r]^{\mu^{n-1}}&{P^{0}}\ar[r]^{\nu^{n-1}}&Y^{n-1}.
 }}
\end{align*}
We have the following commutative diagram for each $\theta^{i}(P^i)$
$$\xymatrix{
X\ar[d]&X^0\ar[r]^{d_X^0}\ar[d]_{\mu^{i-1}d_X^{0,i-2}}&
\cdots\ar[r]&
X^{i-1}\ar[r]^{d_X^{i-1}}\ar[d]_{\mu^{i-1}}&
X^{i}\ar[r]^{d_X^{i}}\ar[d]|{^{\sigma}(\mu^{i-1}d_X^{0,i-2})d_X^{i,n-1}}&
X^{i+1}\ar[r]^{d_X^{i+1}}\ar[d]&
\cdots
\\
\theta^{i}(P^i)\ar[d]&^{\sigma^{-1}}(P^{i})\ar[r]^{\id}\ar[d]|{{^{\sigma^{-1}}}(d_Y^{i, n-1}\nu^i)}&
\cdots\ar[r]^{\id}&
{^{\sigma^{-1}}(P^{i})}\ar[r]^{\omega_{^{\sigma^{-1}}(P^i)}}\ar[d]|{d_Y^{0,i-2}{^{\sigma^{-1}}}(d_Y^{i,n-1}\nu^{i-1})}&
P^{i}\ar[r]^{\id}\ar[d]^{\nu^{i-1}}&
P^{i}\ar[r]^{\id}\ar[d]&
\cdots
\\
Y&Y^0\ar[r]^{d_Y^0}&
\cdots\ar[r]&
Y^{i-1}\ar[r]^{d_Y^{i-1}}&
Y^{i}\ar[r]^{d_Y^{i}}&
Y^{i+1}\ar[r]^{d_Y^{i+1}}&
\cdots
}
$$

One can verify that $(f^0, f^1,\cdots, f^{n-1})$ equals the sum of the aforementioned decompositions, hence it 
factors through $\theta^0(P^0)\oplus\theta^1(P^1)\oplus\cdots\oplus\theta^{n-1}(P^{n-1})$. By reversing the earlier argument, one can prove the ``if" part.
\end{proof}

\section{$N$-module factorizations as modules}\label{sec:module factorizations}
In this section, we realize the category of $n$-fold module factorizations as a module category of a matrix subring. Furthermore, we provide a characterization of Gorenstein projective module factorizations.

\subsection{$N$-fold module factorizations as modules}
We define a subring of $M_{n}(A)$
$$\Gamma_n=\begin{pmatrix}
    A&A&\cdots&A\\
    A\omega& A&\cdots&A\\
    \vdots&\vdots&\ddots&\vdots\\
    A\omega&A\omega&\cdots&A
\end{pmatrix}$$
and define functors
\begin{align*}
    \Phi:& \F_n(A;\omega)\longrightarrow \Gamma_n\Mod \\
         &  X \longmapsto (X^{n-1}, X^{n-2},\cdots, X^0)^T_{(f_{ij})}
\end{align*}
where \[ f_{ij} =\begin{dcases}
d_{X}^{n-j,n-i-1}, & \text{if}~ 1\leq i<j\leq n;\\
d_X^{0,i-1}~{^{\sigma^{-1}}(d_X^{n-j,n-1})}, &\text{if}~ 1\leq j<i\leq n .
\end{dcases} \]
and 
\begin{align*}
    \Psi:  \Gamma_n\Mod &\longrightarrow \F_n(A;\omega) \\
           (X^{n-1}, X^{n-2},\cdots, X^0)^T_{(f_{ij})}&\longmapsto (X^0\stackrel{f_{n-1,n}}{\to}X^1\stackrel{f_{n-2,n-1}}{\to}\cdots\to X^{n-1}\stackrel{^{\sigma}(f_{n, 1})}{\to}{^{\sigma}(X^0)})
\end{align*}
Here, we identify the $A$-module $A\omega\otimes_{A}M$ with $^{\sigma^{-1}}(M)$ via the isomorphism
$$a\omega\otimes m\mapsto ~^{\sigma^{-1}}(^{\sigma^{-1}}(a)m).$$

\begin{prop}\label{prop:module cat}
The functor $\Phi:\F_n(A;\omega)\longrightarrow \Gamma_n\Mod$ is an equivalence, and $\Psi$ is a quasi-inverse of it. In particular, the projective objects of $\F_n(A;\omega)$ are precisely direct summands of $\theta^0(P^0)\oplus\theta^1(P^1)\oplus\cdots\oplus\theta^{n-1}(P^{n-1})$ for projective $A$-modules $P^0,P^1,\cdots,P^{n-1}$.
\end{prop}
\begin{proof}
It is routine to check the following natural isomorphisms
$$\Phi\Psi\simeq {\id}_{\Gamma_n\Mod}~\text{and}~\Psi\Phi\simeq {\id}_{\F_n(A;\omega)}.$$
By the definition of the functor $\Psi$, we obtain the projective objects of $\F_n(A;\omega)$.
\end{proof}

It is known that the category $\F_n(A;\omega)$ is an abelian category with component-wise kernels and cokernels. There is the following diagram
$$\xymatrix{
{\F_{n+1}(A;\omega)}\ar[rr]^{{\pr}_{n+1}^0}
&
&{F_{n}(A;\omega)}\ar^-{{\pr}_{n}^0}[rr]\ar^-{S^n\theta_n^0 S^{-(n-1)}}@/^1.2pc/[ll]\ar_-{\theta_n^0}@/_1.3pc/[ll]
&
&\cdots\ar^-{S^{n-1}\theta_{n-1}^0 S^{-(n-2)}}@/^1.2pc/[ll]\ar_-{\theta^0_{n-1}}@/_1.3pc/[ll]\ar^-{{\pr}_2^0}[rr]
&
&A\Mod\ar^-{S\theta^0_1}@/^1.2pc/[ll]\ar_-{\theta_1^0}@/_1.3pc/[ll]\quad (*)
}$$
Lemma~\ref{lem:adjoint pair} promises us that $(\theta_{n-1}^0\cdots\theta_1^0, {\pr}^0_2\cdots {\pr}_{n}^0)$ and $({\pr}^0_2\cdots {\pr}_{n}^0, S^{n-1}\theta^0_{n-1}\cdots\theta_1^0)$ are adjoint pairs, and so
$${\pr}^0={\pr}^0_2\cdots {\pr}_{n}^0: \F_n(A;\omega)\longrightarrow A\Mod,~X\longmapsto X^0$$
is a Frobenius functor, and so are the functors ${\pr}^i={\pr}^0 S^i, i=1, 2,\cdots, n-1$.

Refer to \cite{M65, CGN99} for Frobenius extensions and to \cite{CDM02} for their generalization.
Let $\C$ and $\D$ be two additive categories. Fix two autoequivalences $\alpha:\C\to \C$ and $\beta:\D\to\D$. Let $F:\C\to \D$ and $G:\D\to \C$ be two additive functors. The pair $(F, G)$ is a {\em Frobenius pair of type $(\alpha, \beta)$} if both $(F, G)$ and $(G, \beta F\alpha)$ are adjoint pairs. So, we have $(\theta^i,pr^i)$ is a Frobenius pair of type $(\id, S^{n-1})$ for $i=0,1,\cdots, n-1$

Define 
\begin{align*}
    \pr:=({\pr}^0, {\pr}^1,\cdots, {\pr}^{n-1}):F_n(A;\omega)&\longrightarrow A\Mod\times A\Mod\times\cdots A\Mod \\
    X&\longmapsto ({\pr}^0(X), {\pr}^1(X),\cdots, {\pr}^{n-1}(X))
\end{align*}
and 
\begin{align*}
    \theta:=\theta^0+\theta^1+\cdots+\theta^{n-1}:A\Mod&\times A\Mod\times\cdots A\Mod\longrightarrow F_n(A;\omega) \\
    (M_1,M_2,\cdots, M_n)&\longmapsto \theta^0{(M_1)}\oplus \theta^1{(M_2)}\oplus\cdots\oplus \theta^{n-1}{(M_n)}
\end{align*}

\begin{prop}\label{prop:faithful frobenius}
The functor
    $$\pr:\F_n(A;\omega)\longrightarrow A\Mod\times A\Mod\times\cdots A\Mod$$
is a faithful Frobenius functor.
\end{prop}
\begin{proof}
According to \cite[Lemma 3.7]{Chen24}, the pair $(\theta, \pr)$ is a Frobenius pair of type $(\id,S^{n-1})$. And it is obvious that $\pr$ is faithful.
\end{proof}

The following lemma indicates that a faithful Frobenius functor preserves Gorensteinn projective objects.

\begin{lem}{\textnormal{(\cite[Theorem 3.2]{CR22})}}\label{lem:frobenius-gorenstein}
  Let $F:\A\to\C$ be a faithful Frobenius functor between two abelian categories with enough projective objects. Then for each $X\in\A$, $X\in\GP(A)$ if and only if $F(X)\in\GP(\C)$.
\end{lem}

\begin{prop}\label{prop:projs}
    The subcategory $\GF_n(A;\omega)$ coincides with $\GP(\F_n(A;\omega))$, and is equivalent to $\Gamma_n\GProj$ as an exact category. In particular, the exact category $\GF_n(A;\omega)$ is Frobenius with projective-injective objects being direct summands of $\theta^0(P^0)\oplus\theta^1(P^1)\oplus\cdots\oplus\theta^{n-1}(P^{n-1})$ for some projective $A$-modules $P^0,P^1,\cdots,P^{n-1}$.
\end{prop}

\begin{proof}
Let $X\in \F_n(A;\omega)$, according to Proposition~\ref{prop:faithful frobenius} and Lemma~\ref{lem:frobenius-gorenstein}, $X$ is Gorenstein projective if and only if its components are Gorenstein projective $A$-modules. And the projective-injective objects are precisely the objects of $\F_n(A;\omega)$, by Proposition~\ref{prop:module cat}, we prove the statement.
\end{proof}

According to Lemma~\ref{lem:p-null-homo}, we know that the homotopy catgeory $\H(\GF_n(A;\omega))$ coincides with the stable category $\underline{\GF}_n(A;\omega)$. 
The following corollary follows directly from 
Proposition~\ref{prop:projs}.

\begin{cor}\label{cor:Gorenstein stable}
    There is a triangle equivalence
    $$\underline{\GF}_n(A;\omega)\longrightarrow \Gamma_n\underline{\GProj}.$$
\end{cor}

\section{The Cokernel functors}\label{sec:cokernel}
In this section, we introduce a cokernel functor from $\F_n(A;\omega)$ to the module category of an $(n-1)\times (n-1)$ upper-triangular matrix ring of $\bar{A}$. By leveraging the characterization of Gorenstein projective objects within this upper-triangular matrix ring as described in the monomorphism category, we establish a non-commutative analog of Eisenbud's correspondence for $n$-fold module factorizations.

Define an $(n-1)\times (n-1)$ upper-triangular matrix ring
$$\Lambda_{n-1}=\begin{pmatrix}
    \bar{A}&\bar{A}&\cdots &\bar{A}\\
    0 &\bar{A}&\cdots&\bar{A}\\
    \vdots&\vdots&\ddots&\vdots\\
    0&0&\cdots&\bar{A}
\end{pmatrix}$$

Denote $\S_{n-1}(\bar{A}\GProj)$ as the monomorphism category (\cite{LZ13}) of $\bar{A}\GProj$, which consists of objects with $(n-2)$ composable monomorphisms:
$$\xymatrix{
X^0\ar[r]^{f^0}& X^1\ar[r]^{f^1}& \cdots\ar[r]^{f^{n-3}}& X^{n-2}
}$$
where $X^i\in \bar{A}\GProj$ and $\Cok(f^i)\in \bar{A}\GProj$. The following result is well-known in the context of finitely generated modules over an Artin algebra \cite{LZ13}. In fact, the proof can be easily extended to infinitely generated modules over any ring (see, for example, \cite{LZHZ20}).

\begin{thm}{\textnormal{(\cite[Theorem 4.1]{LZ13})}}\label{thm:GProj over matrix algebra}
 The category of Gorenstein projective $\Lambda_{n-1}$-module is equivalent to the monomorphism category $\S_{n-1}(\bar{A}\GProj)$.
\end{thm}

Given an $n$-fold module factorization $X=(X^0,...,X^{n-1};d_X^0,...,d_X^{n-1})$, it induces a commutative diagram of morphisms among $n-1$ 2-fold module factorizations.
$$\xymatrix{
X^0\ar@{=}[d]\ar[rr]^{d_X^0}&&
X^1\ar[d]_{d_X^1}\ar[rr]^{d_X^{1,n-1}}&&
^{\sigma}(X^0)\ar@{=}[d]\\
X^0\ar@{=}[d]\ar[rr]^{d_X^{0,1}}&&
X^2\ar[d]_{d_X^2}\ar[rr]^{d_X^{2,n-1}}&&
^{\sigma}(X^0)\ar@{=}[d]\\
\vdots\ar@{=}[d]&&
\vdots\ar[d]_{d_X^{n-2}}&&
\vdots\ar@{=}[d]&&\\
X^0\ar[rr]^{d_X^{0,n-2}}&&
X^{n-1}\ar[rr]^{d_X^{n-1}}&&
^{\sigma}(X^0)
}
$$
This observation is used to define the {\em zeroth cokernel functor}
\begin{align*}
   \Cok^0:\F_n(A;\omega)&\longrightarrow \Lambda_{n-1}\Mod \\
   X&\longmapsto \begin{pmatrix}
    \Cok(d_X^{0,n-2})\\
   \Cok(d_X^{0,n-3})\\
    \vdots\\
    \Cok(d_X^0)
\end{pmatrix}
\end{align*}
The action of $\Cok^0$ on morphisms is standard.

We define the following subcategories of $\F_n(A;\omega)$ and $\GF_n(A;\omega)$
\begin{align*}
    \F^{mp}_n(A;\omega):=&\{X\in\F_n(A;\omega)\mid d_X^i~\text{is mono for}~0\leq i\leq n-2, X^{n-1}~\text{is projective}\}\\
   \G^0\F_n(A;\omega):=&\{X\in\GF_n(A;\omega)\mid X^{n-1}~\text{is projective}\}
\end{align*}
Since each Gorenstein projective $A$-module is $\omega$-torsionfree, hence $\G^{0}\F_{n}(A;\omega)\subseteq \F^{mp}_n(A;\omega)$.

The following lemma differs slightly from \cite[Lemma 2.4]{Chen24}. We need it to prove Lemma~\ref{full and dense}.

\begin{lem}{\textnormal{(Compare with \cite[Lemma 2.4]{Chen24})}}\label{2-fold}
Let $L,M,N$ be $A$-modules with $\omega_{N}=0$ and
$$\xymatrix{
0\ar[r]& L\ar[r]^{f}& M\ar[r]^{g}& N\ar[r]& 0
}$$
be an exact sequence in $A\Mod$. Then there exists a unique morphism $h: M\to {^{\sigma}(L)}$ such that $\omega_{L}=hf$ and $\omega_{M}={^{\sigma}}(f)h$.
\end{lem}

\begin{proof}
We have the following commutative diagram
$$
\xymatrix{
0\ar[r]&
L\ar[rr]^{f}\ar[d]_{\omega_L}&&
M\ar[d]_{\omega_M}\ar[rr]^{b}\ar@{-->}[dll]_{h}&&
N\ar[d]_{\omega_N}\ar[r]&
0\\
0\ar[r]&
^{\sigma}(L)\ar[rr]^{^{\sigma}(f)}&&
^{\sigma}(M)\ar[rr]^{^{\sigma}(g)}&&
^{\sigma}(N)\ar[r]&
0}
.$$
Since $\omega_N=0$, there is a unique $h:M\to {^{\sigma}}(L)$ such that $\omega_{M}={^{\sigma}}(f)h$. Moreover, ${^{\sigma}}(f)hf=\omega_{M}f={^{\sigma}(f)}\omega_L$, because ${^{\sigma}(f)}$ is mono, then $\omega_{L}=hf$.
\end{proof}

\begin{lem}{\label{full and dense}}
The following statements hold.
\begin{enumerate}
    \item The zeroth cokernel functor ${\Cok^0:\F^{mp}_{n}}(A;\omega)\longrightarrow \Lambda_{n-1}\Mod$ is full.
    \item The zeroth cokernel functor ${\Cok^0:\G^{0}\F_{n}}(A;\omega)\longrightarrow \Lambda_{n-1}\GProj$ is full and dense.
\end{enumerate}
\end{lem}
\begin{proof}
(1) Let $X,Y$ be $n$-fold module factorizations in $\F_{n}^{mp}(A;\omega)$ with 
$$\Cok^0(X)=M=(M^i,s^i),~\Cok^0(Y)=N=(N^i, t^i).$$
Let $(f^1,\cdots, f^{n-1}):M\longrightarrow N$ be a morphism in $\Lambda_{n-1}\Mod$. We have the following commutative diagram
$$\xymatrix{
0\ar[r]&
X^0\ar@{-->}[d]_{g^0}\ar[rr]^{d_X^{0,n-1}}&&
X^{n-1}\ar@{-->}[d]_{g^{n-1}}\ar[rr]^{\pi_{X}^{n-1}}&&
M^{n-1}\ar[d]_{f^{n-1}}\ar[r]&
0\\
0\ar[r]&
Y^0\ar[rr]^{d_Y^{0,n-1}}&&
Y^{n-1}\ar[rr]^{\pi_{Y}^{n-1}}&&
N^{n-1}\ar[r]&
0
}
$$
where $g^{n-1}$ exists since $X^{n-1}$ is a projective $A$-module. Moreover, $g^0$ exists. There is also another commutative diagram
$$
\xymatrix{
0\ar[r]&
M^{n-2}\ar[rr]^{s^{n-2}}\ar[d]_{f^{n-2}}&&
M^{n-1}\ar[rr]^{a^{n-1}}\ar[d]_{f^{n-1}}&&
C^{n-1}\ar@{-->}[d]_{h^{n-1}}\ar[r]&
0\\
0\ar[r]&
N^{n-2}\ar[rr]^{t^{n-2}}&&
N^{n-1}\ar[rr]^{b^{n-1}}&&
D^{n-1}\ar[r]&
0}
$$
where $C^{n-1}$ and $D^{n-1}$ are the cokernels of $s^{n-2}$ and $t^{n-2}$, respectively, and $h^{n-1}$ exists uniquely.

By the construction of $\Cok^0$, there are the following two commutative diagrams
$$
\xymatrix{
X^0\ar[rr]^{d_X^{0,n-3}}\ar@{=}[d]&&
X^{n-2}\ar[d]_{d_{X}^{n-2}}\ar[rr]^{\pi_{X}^{n-2}}&&
M^{n-2}\ar[d]_{s^{n-2}}\\
X^0\ar[rr]^{d_X^{0,n-2}}&&
X^{n-1}\ar[d]\ar[rr]^{\pi_{X}^{n-1}}&&
M^{n-1}\ar[d]_{a^{n-1}}\\
&&
C^{n-1}\ar@{=}[rr]&&
C^{n-1}}
$$
and
$$
\xymatrix{
Y^0\ar[rr]^{d_Y^{0,n-3}}\ar@{=}[d]&&
Y^{n-2}\ar[d]_{d_{Y}^{n-2}}\ar[rr]^{\pi_{Y}^{n-2}}&&
N^{n-2}\ar[d]_{t^{n-2}}\\
Y^0\ar[rr]^{d_Y^{0,n-2}}&&
Y^{n-1}\ar[d]\ar[rr]^{\pi_{Y}^{n-1}}&&
N^{n-1}\ar[d]_{b^{n-1}}\\
&&
D^{n-1}\ar@{=}[rr]&&
D^{n-1}}
$$
Since $$h^{n-1}a^{n-1}\pi_{X}^{n-1}=b^{n-1}f^{n-1}\pi_{X}^{n-1}=b^{n-1}\pi_{Y}^{n-1}g^{n-1},$$
there exists a unique morphism $g^{n-2}$ such that the following diagram commutes
$$
\xymatrix{
0\ar[r]&
X^{n-2}\ar[rr]^{d_{X}^{n-2}}\ar@{-->}[d]_{g^{n-2}}&&
X^{n-1}\ar[rr]^{a^{n-1}\pi_{X}^{n-1}}\ar[d]_{g^{n-1}}&&
C^{n-1}\ar[d]_{h^{n-1}}\ar[r]&
0\\
0\ar[r]&
Y^{n-2}\ar[rr]^{d_{Y}^{n-2}}&&
Y^{n-1}\ar[rr]^{b^{n-1}\pi_{Y}^{n-1}}&&
D^{n-1}\ar[r]&
0}
$$
Since 
\begin{align*}
    t^{n-2}f^{n-2}\pi_{X}^{n-2}&=f^{n-1}s^{n-2}\pi_{X}^{n-2}\\
                             &=f^{n-1}\pi_{X}^{n-1}d_{X}^{n-2}\\
                             &=\pi_{Y}^{n-1}g^{n-1}d_{X}^{n-2}\\
                             &=\pi_{Y}^{n-1}d_{Y}^{n-2}g^{n-2}\\
                             &=t^{n-2}\pi_{Y}^{n-2}g^{n-2}
\end{align*}
Since $t^{n-2}$ is a monomorphism, then $f^{n-2}\pi_{X}^{n-2}=\pi_{Y}^{n-2}g^{n-2}$. On the other hand, we have a diagram
$$
\xymatrix{
X^0\ar[d]_{g^0}\ar[rr]^{d_X^{0,n-2}}&&
X^{n-2}\ar[rr]^{d_{X}^{n-2}}\ar[d]_{g^{n-2}}&&
X^{n-1}\ar[d]_{g^{n-1}}\\
Y^0\ar[rr]^{d_Y^{0,n-2}}&&
Y^{n-2}\ar[rr]^{d_{Y}^{n-2}}&&
Y^{n-1}
}
$$
where the outer square and right square are commutative. Since $d_{Y}^{n-2}$ is a monomorphism, the left square is commutative. Therefore, we have the following commutative diagram
$$
\xymatrix{
0\ar[r]&
X^0\ar[rr]\ar[d]_{g^0}&&
X^{n-2}\ar[d]_{g^{n-2}}\ar[rr]^{\pi_{X}^{n-2}}&&
M^{n-2}\ar[d]_{f^{n-2}}\ar[r]&
0\\
0\ar[r]&
Y^0\ar[rr]&&
Y^{n-2}\ar[rr]^{\pi_{Y}^{n-2}}&&
N^{n-2}\ar[r]&
0.}
$$
By using recursive methods, one can construct $g=(g^0,g^1,\cdots,g^{n-1})$. We have the following diagram
$$
\xymatrix{
X^{n-1}\ar[d]_{g^{n-1}}\ar[rr]^{d_{X}^{n-1}}&&
^{\sigma}(X^0)\ar[d]_{^{\sigma}(g^0)}\ar[rr]&&
^{\sigma}(X^{n-1})\ar[d]_{^{\sigma}(g^{n-1})}\\
Y^{n-1}\ar[rr]^{d_{Y}^{n-1}}&&
{^{\sigma}(Y^0)}\ar[rr]^{{^{\sigma}(d_Y^{0,n-2})}}&&
^{\sigma}(Y^{n-1}),}
$$
where the outer square and the right square are commutative. Since ${^{\sigma}(d_Y^{0,n-2})}$ is a monomorphism, the left square commutative. Therefore, $g$ is a morphism between module factorizations and $\Cok^0(g)=f$.

(2) By applying Theorem \ref{thm:GProj over matrix algebra} and the construction of $\Cok^0$, we have the functor is well-defined. Due to (1), it suffices to prove the functor $\Cok^0$ is dense. Because of Theorem~\ref{thm:GProj over matrix algebra}, we may identify the Gorenstein projective $\Lambda_{n-1}$-modules as objects in $\S_{n-1}(\bar{A}\GProj)$.
Assume
$$\xymatrix{
G=G^1\ar[r]^{\quad f^1}& G^2\ar[r]^{f^2}& \cdots\ar[r]^{f^{n-3}}& G^{n-2}\ar[r]^{f^{n-2}}& G^{n-1}
}$$
is an object in $\S_{n-1}(\bar{A}\GProj)$. Let 
$$\xymatrix{
0\ar[r]&X^0\ar[rr]^{d_X^{0,n-2}}&& X^{n-1}\ar[rr]^{\mu^{n-1}}&& G^{n-1}\ar[r]& 0
}$$
be an exact sequence in $A\Mod$
with $X^{n-1}\in A\Proj$ and $\mu^{n-1}$ a projective cover, according to Theorem \ref{GP under syzygy}, $X^0\in A\GProj$. Begin at $X^0$ and $X^{n-1}$, then we construct $X^{n-2}$ and a monomorphism $d_X^{n-2}:X^{n-2}\longrightarrow X^{n-1}$ which satisfy the following commutative diagram
$$
\xymatrix{
& &0\ar[d]&0\ar[d]&\\
& &X^0\ar@{-->}[r]\ar[d]_{d_X^{0,n-2}}&X^{n-2}\ar@{-->}[d]^{d_X^{n-2}}&\\
& &X^{n-1}\ar[d]_{\mu^{n-1}}\ar@{=}[r]&X^{n-1}\ar[d]&\\
0\ar[r]&G^{n-2}\ar[r]^{f^{n-2}}&G^{n-1}\ar[d]\ar[r]&\Cok(f^{n-2})\ar[d]\ar[r]&0\\
&&0&0}
$$
Since $\Cok(f^{n-2})\in \bar{A}\GProj$, $X^{n-2}\in A\GProj$ by Theorem~\ref{GP under syzygy}. Note that the morphism $d^{0,n-2}$ from $X^0$ to $X^{n-1}$ factors through $d_X^{n-2}$.

It follows from the Snake Lemma, there is a short exact sequence 
$$\xymatrix{
0\ar[r]& X^0\ar[rr]^{d^{0,n-2}}&& X^{n-2}\ar[rr]^{\mu^{n-2}}&& G^{n-2}\ar[r]& 0.
}$$
By using recursive methods, we obtain a sequence of $n-1$ monomorphisms
$$\xymatrix{
X^0\ar[r]^{d_X^0}& X^1\ar[r]^{d_X^1}& \cdots\ar[r]^{d_X^{n-3}}& X^{n-2}\ar[r]^{d_X^{n-2}}& X^{n-1},
}$$
such that $\Cok(d_X^{0,i})=G^{i+1}$ for $i=0,\cdots, n-2$. By Lemma~\ref{2-fold}, there exists a unique morphism $d_X^{n-1}:X^{n-1}\longrightarrow ~{^{\sigma}(X^0)}$ such that $d_X^{0,n-1}=\omega_{X^0}$ and $^{\sigma}(d_X^{0,n-2})d_X^{n-1}=\omega_{X^{n-1}}$.
We need to check this is an $n$-fold module factorization, i.e., the composition, denoted by $\phi$, of 
$$\xymatrix{
X^i\ar[r]^{d_X^i\quad}& X^{i+1}\ar[r]^{d_X^{i+1}\quad\quad}& X^{i+2}\cdots X^{n-1}\ar[r]^{\quad\quad d_X^{n-1}}& ~^{\sigma}(X^0)\ar[r]^{^{\sigma}(d_X^0)}&\cdots \ar[r]^{^{\sigma}(d_X^{i-1})}& ~^{\sigma}(X^i)
}$$
is $\omega_{X^{i}}$. By composing the above sequence with $^{\sigma}(d_X^{i,n-2})$, we have
$${^{\sigma}}(d_X^{i,n-2})\phi=\omega_{X^{n-1}}d_X^{i,n-2}={^{\sigma}}(d_X^{i,n-2})\omega_{X^i}.$$
Since ${^{\sigma}}(d_X^{i,n-2})$ is a monomprphism, then $\phi=\omega_{X^i}$.
\end{proof}

\begin{lem}{\label{faithful}}
    Let $(f^0,f^1,\cdots,f^{n-1}):X\to Y$ be a morphism in $\F^{mp}_n(A;\omega)$. Then
    \begin{enumerate}
        \item $\Cok^0{(f^0,f^1,\cdots,f^{n-1})}=0$ if and only if $(f^0,f^1,\cdots,f^{n-1})$ factors through $\theta^0(P)$ for some projective $A$-module $P$;
        \item The morphism $\Cok^0{(f^0,f^1,\cdots,f^{n-1})}$ factors through a projective $\Lambda_{n-1}$-module if and only if $(f^0,f^1,\cdots,f^{n-1})$ is p-null-homotopical.
    \end{enumerate}
\end{lem}
\begin{proof}
 The idea of the proof below is essentially from \cite[Lemma 5.4]{Chen24}.

(1) The ``if" part is obvious. For the ``only if" part, 
assume $$\Cok^0{(f^0,f^1,\cdots,f^{n-1})}=0,$$
then there exist $h^i:X^i\to Y^{0}$ for $i=1,\cdots, n-1$ such that
$d_Y^{0,i-1} h^i=f^i$

$$\xymatrix
{X^0\ar[rr]^{d_X^0}\ar[dd]_{f^0}&&X^1\ar[rr]^{d_X^1}\ar[dd]^{ f^1}\ar[ddll]_{h^1}&&X^2\ar[r]^{d_X^2}\ar[dd]^{f^2}\ar[ddllll]_{h^2\quad}&\cdots\ar[r]^{d_X^{n-2}}&X^{n-1}\ar[r]^{d_X^{n-1}}\ar[dd]^{f^{n-1}}\ar[ddllllll]_{h^{n-1}}&{^\sigma(X^0)}\ar[dd]^{^\sigma(f^0)}
\\
\\
Y^0\ar[rr]^{\quad d_Y^0}&&Y^1\ar[rr]^{d_Y^1}&&Y^2\ar[r]^{d_Y^2}&\cdots\ar[r]^{d_Y^{n-2}}&Y^{n-1}\ar[r]^{d_X^{n-1}}&{^\sigma(Y^0)}
}$$

Moreover, we have the following commutative diagram
$$\xymatrix{
X^0\ar[r]^{d_X^0}\ar[d]_{d_X^{0,n-2}}&X^1\ar[r]^{d_X^1}\ar[d]_{d_X^{1,n-2}}&X^2\ar[r]^{d_X^2}\ar[d]_{d_X^{2,n-2}}&\cdots\ar[r]^{d_X^{n-2}}&X^{n-1}\ar[r]^{d_X^{n-1}}\ar[d]_{\id}&{^\sigma(X^0)}\ar[d]^{^\sigma(d_X^{0,n-2})}
\\
X^{n-1}\ar[r]^{\id}\ar[d]_{h^{n-1}}&X^{n-1}\ar[r]^{\id}\ar[d]_{d_Y^0 h^{n-1}}&X^{n-1}\ar[r]^{\id}\ar[d]_{d_Y^{0,1} h^{n-1}}&\cdots\ar[r]^{\id}&X^{n-1}\ar[r]^{\omega_{X^{n-1}}}\ar[d]_{d_Y^{0,n-2} h^{n-1}}&X^{n-1}\ar[d]^{^{\sigma}{(h^{n-1})}}
\\
Y^0\ar[r]^{d_Y^0}&Y^1\ar[r]^{d_Y^1}&Y^2\ar[r]^{d_Y^2}&\cdots\ar[r]^{d_Y^{n-2}}&Y^{n-1}\ar[r]^{d_Y^{n-1}}&{^\sigma(Y^0)}
}$$
Since 
$$d_Y^{i,n-2}f^i=f^{n-1}d_X^{i,n-2}=d_Y^{0,n-2} h^{n-1} d_X^{i,n-2}$$
and $d_Y^0, \cdots, d_Y^{n-2}$ are monomorphic, so we obtain $f^i=d_Y^{0,i-1} h^{n-1} d_X^{i,n-2}$. 
Therefore $(f^0,f^1,\cdots,f^{n-1})$ factors through $\theta^0(X^{n-1})$ with $X^{n-1}$ a projective $A$-module.

(2) For the ``if" part, since $(f^0,f^1,\cdots,f^{n-1})$ is p-null-homotopical, then it factors through $\theta^0(P^0)\oplus\theta^1(P^1)\oplus\cdots\oplus\theta^{n-1}(P^{n-1})$ for some projective $A$-modules $P^0,P^1,\cdots,P^{n-1}$. Note that 
$$\Cok^0(\theta^0(P^0))=0, 
\Cok^0(\theta^1(P^1))=\begin{pmatrix}
    P^1/{\omega P^1}\\
    \vdots\\
    P^1/{\omega P^1}
\end{pmatrix},
\Cok^0(\theta^2(P^2))=\begin{pmatrix}  
    P^2/{\omega P^2}\\
    \vdots\\
    P^2/{\omega P^2}\\
        0
\end{pmatrix},\cdots
$$
are projective $\Lambda_{n-1}$-modules. Hence $$\Cok^0(f^0,f^1,\cdots,f^{n-1})$$ factors through a projective $\Lambda_{n-1}$-module $\oplus_{i=0}^{n-1}\Cok^0(\theta^i(P^i))$.

For the “only if " part, assume $\Cok^0(f^0,f^1,\cdots,f^{n-1})$ factors through a projective $\Lambda_{n-1}$-module thus a free $\Lambda_{n-1}$-module $F$. Note that, as $\Lambda_{n-1}$-modules
$$\Cok^0(\theta^{1}(A)\oplus\cdots\oplus \theta^{n-1}(A))\simeq \Lambda_{n-1}.$$ In other words, there exists a trivial module factorization $Q$ such that $\Cok^0(Q)\simeq F$. Since $\Cok^0$ is full on the subcategory $\F_n^{mp}(A;\omega)$ by Lemma~\ref{full and dense}(1), there exist morphisms
$$g=(g^0, g^1, \cdots, g^{n-1}):X\longrightarrow Q~\text{and}~h=(h^0, h^1, \cdots, h^{n-1}):Q\longrightarrow Y$$ such that $\Cok^0(f-hg)=0$. By $(1)$, $f-hg$ factors through $\theta^0(X^{n-1})$ with $X^{n-1}$ a projective $A$-module. 
Note that $hg$ factors through $Q$, thus $f$ is $p$-null homotopic by Lemma~\ref{lem:p-null-homo}.
\end{proof}

Define a class of morphisms in $\F_n(A;\omega)$
$$\I:=\{f\in\Mor(\F_n(A;\omega))\mid f~\text{factors through}~\theta^0{(P)}~\text{for some}~P\in A\Proj\}$$

The following corollary and theorem are non-commutative versions of Eisenbud's theorem(\cite{E80,H21}). In the context of hypersurface rings, they are proved in \cite[Lemma 5.1.2, Theorem 5.2.2]{T21}.

\begin{cor}
The induced functor $\Cok^0: \G^0\F_n(A;\omega)/\I\longrightarrow \Lambda_{n-1}\GProj$ is an equivalence.
\end{cor}
\begin{proof}
It follows directly from Lemma~\ref{full and dense} (2) and   Lemma~\ref{faithful} (1).
\end{proof}

The category $\G^0\F_n(A;\omega)$ is an exact subcategory of $\GF_n(A;\omega)$ which is closed under extension, syzygy, (relative) cosyzygy, and contains the projective-injective objects. Consequently, $\G^0\F_n(A,\omega)$ emerges as a Frobenius exact category with a triangulated stable category denoted as $\underline{\G^0\F}_n(A,\omega)$.

\begin{thm}\label{thm:cok equiv 1}
    The zeroth cokernel functor $\Cok^0$ induces a triangular equivalence
    $$\Cok^0: \underline{\G^0\F}_n(A;\omega)\longrightarrow \Lambda_{n-1}\underline{\GProj}.$$
\end{thm}

\begin{proof}
Since $\Cok^0$ is exact and preserves projective-injective objects, it is a well-defined triangulated functor. By Lemma \ref{faithful}, the induced functor $\Cok^0$ is faithful. 
Combined with Lemma \ref{full and dense}, the statement follows directly.
\end{proof}

An $n$-fold module factorization $X$ is called an {\em $n$-fold matrix factorization} if each $X^i$ is a finitely generated projective $A$-module. By $\MF_n(A;\omega)$ we denote the full subcategory of $\F_n(A;\omega)$ consisting of $n$-fold matrix factorizations. We denote $\Lambda_{n-1}{\Gproj}^{<+\infty}$ the full subcategory of $\Lambda_{n-1}{\Gproj}$ comprising objects where each component possesses finite projective dimension as an $A$-module. Note that $\Lambda_{n-1}{\Gproj}^{<+\infty}$ is extension-closed, contains projective objects of $\Lambda_{n-1}{\Gproj}$, and is a Frobenius exact category.

\begin{thm}\label{thm:matrix factor}
The functor $\Cok^0$ induces a triangular equivalence
$$\Cok^0: \underline{\MF_n}(A;\omega)\longrightarrow \Lambda_{n-1}\underline{\Gproj}^{<+\infty}$$
\end{thm}
\begin{proof}
It follows from \cite[Theorem 7.2]{Chen24} and the definition of the zeroth cokernel (see Section~\ref{sec:cokernel}) that $\Cok^0$ is well-defined, and Theorem~\ref{thm:cok equiv 1} implies that $\Cok^0$ is fully faithful. It remains to prove the denseness of $\Cok^0$. Let 
$$\xymatrix{
G=G^0\ar[r]^{\quad f^0}& G^1\ar[r]^{f^1}& \cdots\ar[r]^{f^{n-4}}& G^{n-3}\ar[r]^{f^{n-3}}& G^{n-2}
}$$
be an object in $\Lambda_{n-1}{\Gproj}^{<+\infty}$, due to Lemma~\ref{full and dense} and its proof, $G$ admits a preimage $X$ of which each component as an $A$-module is finitely generated Gorenstein projective and has finite projective dimension. It is known to all that a Gorenstein projective module with finite projective dimension is projective. Hence, $X$ is in $\MF_n(A;\omega)$. We finish the proof.
\end{proof}

\section{Recollements}\label{sec:recollement}
The concept of recollements of triangulated categories,  initially introduced by Beilinson, Bernstein, and Deligne \cite{BBD}, plays a crucial role in representation theory. In this section, we will explore the occurrence of recollements within the triangulated stable categories discussed in the previous sections. 

\begin{defn}
Let $\T$, $\X$ and $\Y$ be triangulated categories.
We say that $\T$ is a {\em recollement} (\cite{BBD}) of $\X$
and $\Y$ if there is a diagram of six triangle functors
$$\xymatrix{\X\ar^-{i_*=i_!}[rr]&&\T\ar^-{j^!=j^*}[rr]
\ar^-{i^!}@/^1.2pc/[ll]\ar_-{i^*}@/_1.6pc/[ll]
&&\Y\ar^-{j_*}@/^1.2pc/[ll]\ar_-{j_!}@/_1.6pc/[ll]}$$ 
such that
\begin{enumerate}
    \item $(i^*,i_*),(i_!,i^!),(j_!,j^!)$ and $(j^*,j_*)$ are adjoint
pairs;
    \item $i_*,j_*$ and $j_!$ are fully faithful functors;
    \item $i^!j_*=0$; and
    \item for each object $T\in\T$, there are two triangles in
$\T$
$$
i_!i^!(T)\longrightarrow T\longrightarrow j_*j^*(T)\longrightarrow i_!i^!(T)[1],
$$
$$
j_!j^!(T)\longrightarrow T\longrightarrow i_*i^*(T)\longrightarrow j_!j^!(T)[1].
$$
\end{enumerate}
\end{defn}

The adjoint triples $(*)$ in section~\ref{sec:module factorizations} (after Proposition~\ref{prop:module cat}) can be restricted to subcategories $\GF_{n+1}(A;\omega), \G^0F_{n+1}(A;\omega)$ and $\MF_{n}(A;\omega)$.

\begin{lem}\label{lem:res to stable G0F}
    The face functor $\theta^{i-1}_i$ induces a fully faithful functor 
$$\theta^{i-1}_i:{\underline{\F}_{i}(A;\omega)}\longrightarrow {\underline{\F}_{i+1}(A;\omega)}$$
which restricts to triangle functors 
$$\xymatrix{
{\underline{\G\F}_{i}(A;\omega)}\ar[rr]^{\theta^{i-1}_i}&&{\underline{\G\F}_{i+1}(A;\omega)}
\\
{\underline{\G^0\F}_{i}(A;\omega)}\ar@{^(->}[u]\ar[rr]^{\theta^{i-1}_i}&& {\underline{\G^0\F}_{i+1}(A;\omega)}\ar@{^(->}[u]
\\
{\underline{\MF}_{i}(A;\omega)}\ar@{^(->}[u]\ar[rr]^{\theta^{i-1}_i}&& {\underline{\MF}_{i+1}(A;\omega)}\ar@{^(->}[u]
}$$
\end{lem}
\begin{proof}
It is routine to check that $\theta^{i-1}_i$ is well-defined. It is obvious that $\theta^{i-1}_i$ is full. 
We prove $\theta^{i-1}_i$ is faithful. Let $X, Y\in \F_{i}(A;\omega)$ and $f:X\to Y$ satisfy $\theta^{i-1}_i(f)=0$ in ${\underline{\G\F}_{i+1}(A;\omega)}$. Assume that $\theta^{i-1}_i(f)=(f^0, f^1, \cdots, f^{i-1}, f^{i-1})$ has a homotopy $(h^0, h^1,\cdots, h^{i-1}, h^i)$. Then $(f^0, f^1, \cdots, f^{i-1})$ has a homotopy 
$$(h^0, h^1,\cdots, ^{\sigma^{-1}}(d_Y^{i-1})h^{i-1}+h^i)$$
this implies that $f=0$ in ${\underline{\G\F}_{i}(A;\omega)}$.
\end{proof}

\begin{thm}\label{thm:recollement}
    There are the following recollements for $1\leq k\leq n-1$
$$\xymatrix{
{\underline{\G^0\F}_{n-k+1}(A;\omega)}\ar[rr]^{\inc}
&
&
{\underline{\GF}_{n}(A;\omega)}\ar[rr]^{{\pr}_{k+1}^0\cdots {\pr}_n^0}\ar@/^1.2pc/[ll]^{\inc_{\rho}}\ar@/_1.3pc/[ll]_{\inc_{\Lambda_{n-1}}}
&
&{\underline{\GF}_{k}(A;\omega)}\ar^-{S^n\theta_n^0\cdots\theta_k^0 S^{-(k-1)}}@/^1.2pc/[ll]\ar_-{\theta_n^0\cdots\theta_{k+1}^0}@/_1.3pc/[ll]
}$$  
where $\inc=S^{-1}\theta^{n-2}_{n-1}\cdots \theta^{n-k+1}_{n-k+2}\theta^{n-k}_{n-k+1}$. The recollements restrict to recollements
$$\xymatrix{
{\underline{\G^0\F}_{n-k+1}(A;\omega)}\ar[rr]^{\inc}
&
&
{\underline{\G^0\F}_{n}(A;\omega)}\ar[rr]^{{\pr}_{k+1}^0\cdots {\pr}_n^0}\ar@/^1.2pc/[ll]^{\inc_{\rho}}\ar@/_1.3pc/[ll]_{\inc_{\Lambda_{n-1}}}
&
&{\underline{\G^0\F}_{k}(A;\omega)}\ar^-{S^n\theta_n^0\cdots\theta_k^0 S^{-(k-1)}}@/^1.2pc/[ll]\ar_-{\theta_n^0\cdots\theta_{k+1}^0}@/_1.3pc/[ll]
}$$
and 
$$\xymatrix{
{\underline{\MF}_{n-k+1}(A;\omega)}\ar[rr]^{\inc}
&
&
{\underline{\MF}_{n}(A;\omega)}\ar[rr]^{{\pr}_{k+1}^0\cdots {\pr}_n^0}\ar@/^1.2pc/[ll]^{\inc_{\rho}}\ar@/_1.3pc/[ll]_{\inc_{\Lambda_{n-1}}}
&
&{\underline{\MF}_{k}(A;\omega).}\ar^-{S^n\theta_n^0\cdots\theta_k^0 S^{-(k-1)}}@/^1.2pc/[ll]\ar_-{\theta_n^0\cdots\theta_{k+1}^0}@/_1.3pc/[ll]
}$$
\end{thm}

\begin{proof}
For the first recollement, it suffices to show 
 $$\Ker({\pr}_{k+1}^0\cdots {\pr}_n^0)\simeq\underline{\G^0\F}_{n-k+1}(A;\omega).$$

Let $X$ be a non-zero object in $\Ker({\pr}_{k+1}^0\cdots {\pr}_n^0)$, then there exists a projective object in ${{\G^0\F}_{k}(A;\omega)}$
$$\xymatrix{
Q:=Q^0\ar[r]^{\quad d_Q^0}& Q_1\ar[r]^{d_Q^1}&\cdots\ar[r]& Q^{k-1}\ar[r]^{d_Q^{k-1}}& {^{\sigma}}(Q^{0})
}$$
such that in ${{\G^0\F}_{k}(A;\omega)}$
 $$Q\oplus {\pr}_{k+1}^0\cdots {\pr}_n^0(X)\simeq\theta^0(P^0)\oplus \theta^1(P^1)\oplus \cdots\oplus \theta^{k-1}(P^{k-1})$$
for some projective $A$-modules $P^0, P^1$, $\cdots,$ $P^{k-1}$. 

Define an object in ${{\G^0\F}_{n}(A;\omega)}$ as
$$\xymatrix{
Q':=Q^0\ar[r]^{\quad\id}&  Q^0\ar[r]^{\id}&\cdots\ar[r]^{\id}& Q^0\ar[r]^{d_Q^0}& Q^1\ar[r]^{d_Q^1}& \cdots\ar[r]^{d_Q^{k-2}}& Q^{k-1}\ar[r]^{d_Q^{k-1}}& {^{\sigma}}(Q^{0})
}$$
Then in ${{\G^0\F}_{n}(A;\omega)}$
$${\pr}_{k+1}^0\cdots {\pr}_n^0(X\oplus Q')\simeq \theta^0(P^0)\oplus \theta^1(P^1)\oplus \cdots\oplus \theta^{k-1}(P^{k-1}).$$
By easy checking, all $\theta^i(P^i)$ have only projective preimages except $\theta^1(P^1)$. So, in the stable category, $X$ must be of the following form 
 $$(*)\quad
 {\xymatrix{
 {^{\sigma^{-1}}(P)}\ar[r]^{\quad d_X^0}& X^1\ar[r]^{d_X^1}&\cdots\ar[r]^{d_X^{n-k-1}}& X^{n-k}\ar[r]^{\quad d_X^{n-k}}& P\ar[r]^{\id\quad}& \cdots \ar[r]^{\id}& P.
 }}$$
Hence we get a functor $\Phi:\Ker({\pr}_{k+1}^0\cdots {\pr}_n^0)\longrightarrow \underline{\G^0\F}_{n-k+1}(A;\omega)$ which takes the module factorization 
$(*)$ to 
$$\xymatrix{
X^1\ar[r]^{d_X^1}&\cdots\ar[r]^{d_X^{n-k-1}}& X^{n-k}\ar[r]^{\quad d_X^{n-k}}& P\ar[r]^{^\sigma(d_X^{0})\quad}& {^\sigma{(X^1)}}.
}$$
It is routine to check that $\inc=S^{-1}\theta^{n-2}_{n-1}\cdots \theta^{n-k+1}_{n-k+2}\theta^{n-k}_{n-k+1}$ is a quasi-inverse of the functor $\Phi$.
\end{proof}

Setting $n=2$ and $k=1$, Theorem~\ref{thm:recollement} implies \cite[Proposition 6.3]{Chen24}. The assertions in the previous theorem can be reformulated in alternative ways, where the intermediate recollement directly follows from \cite[Corollary 4.4]{LZHZ20}.

\begin{cor}\label{cor:recollement}
    There is the following recollement for $1\leq k<n$
    
$$\xymatrix{
\Lambda_{n-k}\underline{\GProj}\ar[rr]_{\quad}^{\quad}
&
&
\Gamma_n\underline{\GProj}\ar[rr]^{~~~~}\ar@/^1.2pc/[ll]\ar@/_1.3pc/[ll]_{\quad}^{\quad}
&
&\Gamma_k\underline{\GProj}\ar@/^1.2pc/[ll]_{\quad}^{\quad}\ar@/_1.3pc/[ll]_{\quad}^{\quad}
}$$
which can restrict to recollements

$$\xymatrix{
\Lambda_{n-k}\underline{\GProj}\ar[rr]
&
&
\Lambda_{n-1}\underline{\GProj}\ar[rr]\ar@/^1.2pc/[ll]\ar@/_1.3pc/[ll]_{\quad}^{\quad}
&
&\Lambda_{k-1}\underline{\GProj}\ar@/^1.2pc/[ll]_{\quad}^{\quad}\ar@/_1.3pc/[ll]_{\quad}^{\quad}
}$$
and

$$\xymatrix{
\Lambda_{n-k}\underline{\Gproj}^{<+\infty}\ar[rr]_{\quad}^{\quad}
&
&
\Lambda_{n-1}\underline{\Gproj}^{<+\infty}\ar[rr]_{\quad}^{\quad}\ar@/^1.2pc/[ll]\ar@/_1.3pc/[ll]_{\quad}^{\quad}
&
&\Lambda_{k-1}\underline{\Gproj}^{<+\infty}.\ar@/^1.2pc/[ll]_{\quad}^{\quad}\ar@/_1.3pc/[ll]
}$$
\end{cor}

\vspace{2em}
{\bf Acknowledgments:} Yaohua Zhang is supported by the National Natural Science Foundation of China (No. 12401044).

\vspace{2em}
{\bf Declaration of Interest:} All authors disclosed no relevant relationships.


\begin{thebibliography}{99}
\bibitem{BDFIK16}
M. Ballard, D. Deliu, D. Favero, M.U. Isik, and L. Katzarkov,
\newblock{\em Resolutions in factorization categories,}
\newblock Adv. Math., 295 (2016), 195-249.

\bibitem{BFK14}
M. Ballard, D. Favero, and L. Katzarkov,
\newblock{\em A category of kernels for equivariant factorizations and its implications for Hodge theory,}
\newblock Publ. Math. IHES 120 (2014), 1-111.

\bibitem{BBD}
A. A. Beilinson, J. Bernstein, and P. Deligne,
\newblock {\em Faisceaux perverse (French), Analysis and topology on singular spaces,}
\newblock Soc. Math. France, Paris, Asterisque, vol. {\bf 100} (1982),  5-171.

\bibitem{BE19}
P. A. Bergh, and K. Erdmann, 
\newblock{\em Matrix factorizations for quantum complete intersections,}
\newblock J. Homotopy Relat. Struct., 14(4) (2019), 863-880.


\bibitem{BJ23}
P. A. Bergh, and D .A. Jorgensen,
\newblock{\em Categorical matrix factorizations,}
\newblock Ann. K-Theory, 8 (3) (2023), 355-378.

\bibitem{B21}
R. O. Buchweitz,
\newblock{\em Maximal Cohen-Macaulay modules and Tate-cohomology over Gorenstein rings, with appendices by L.L. Avramov, B. Briggs, S.B. Iyengar, and J.C. Letz,}
\newblock Math. Surveys and Monographs 262, Amer. Matn. Soc., 2021.

\bibitem{CDM02}
S. Caenepeel, E. De Groot, and G. Militrau,
\newblock {\em Frobenius functors of the second kind,}
\newblock Comm. Alg., 30 (11) (2002), 5359-5391.

\bibitem{CCKM}
T. Cassidy, A. Conner, E. Kirkman, and W. F. Moore,
\newblock {\em Periodic free resolutions from twisted matrix factorizations,}
\newblock J. Alg., 455 (2016), 137-163.

\bibitem{C09}
N. Carqueville,
\newblock{\em Matrix factorizations and open topological string theory,}
\newblock JHEP 07 (2009), 005.

\bibitem{CGN99}
F. Casta\~no Iglesias, J. G\'omez Torrecilias, and C. N\u{a}st\u{a}sescu, 
\newblock{\em Frobenius functors: applications,}
\newblock Comm. Alg., 27 (10) (1999), 4879-4900.

\bibitem{CKMW}
A. Conner, E. Kirkman, W. F. Moore, and C. Walton,
\newblock {\em Noncommutative Kn\"orrer periodicity and noncommutative Kleinian singularities,}
\newblock J. Alg., 540 (2019), 234-273.

\bibitem{Chen24}
X. W. Chen,
\newblock {\em Module factorizations and Gorenstein projective modules,}
\newblock preprint (2024), arXiv: 2402.11613.

\bibitem{CR22}
X. W. Chen, and W. Ren,
\newblock {\em Frobenius functors and Gorenstein homological properties,}
\newblock J. Alg., 610 (2022), 18-37.

\bibitem{E80}
D. Eisenbud,
\newblock{\em Homological algebra of a complete intersection, with an application to group representations,}
\newblock{Trans. Amer. Math. Soc., 1 (1980), 35-64.}

\bibitem{EP21}
D. Eisenbud, and I. Peeva,
\newblock{\em Layered resolutions of Cohen-Macaulay modules,}
\newblock J. Eur. Math. Soc., 23 (2021), 845-867.

\bibitem{H21}
E. Hopkins,
\newblock{\em N-Fold Matrix Factorizations,}
\newblock Phd Thesis, 2021.

\bibitem{KST07}
H. Kajiura, K. Saito, and A. Takahashi,
\newblock{\em Matrix factorizations and representations of quivers II: Type ADE case,}
\newblock Adv. Math. 211 (1) (2007), 327-362.

\bibitem{KR08}
M. Khovanov, and L. Rozansky,
\newblock{\em Matrix factorizations and link homology,}
\newblock Fund. Math. 199 (2008), 1-91.

\bibitem{LZHZ20}
H. H. Li, Y. F. Zheng, J. S. Hu, and H. Y. Zhu,
\newblock{\em Gorenstein projective modules and recollements
over triangular matrix rings,}
\newblock Comm. Alg., 48 (11) (2020), 4932–4947.

\bibitem{LZ13}
X. H. Luo, and P. Zhang,
\newblock{\em Monic representations and Gorenstein-projective modules,}
\newblock Paci. J. Math., 264 (1) (2013), 163-194.

\bibitem{MU}
I. Mori, and K. Ueyama,
\newblock{\em Noncommutative matrix factorizations with an application to skew exterior algebras,}
J. Alg., 586 (2021), 1053-1087.

\bibitem{M65}
K. Morita,
\newblock {\em Adjoint pairs of functors and Frobenius extensions,}
\newblock Sci. Rep. Tokyo Kyoiku Daigaku, Sec. A, 9 (1965), 40-71.

\bibitem{O04}
D. Orlov,
\newblock{\em Triangulated categories of singularities and D-branes in Landau-Ginzburg models,}
\newblock Trudy Steklov Math. Inst. 204 (2004), 240-262.

\bibitem{PV11}
A. Polishchuk, and A. Vaintob,
\newblock{\em Matrix factorizations and singularity categories for stacks,}
\newblock Ann. Inst. Fourier 61 (2011), 2609-2642.

\bibitem{T21}
T. Tribone,
\newblock{\em Matrix factorizations with more than two factors,}
\newblock preprint (2021), arXiv:2102.06819v1.

\bibitem{Y16}
X. Yu,
\newblock{\em The triangular spectrum of matrix factorizations is the singular locus,}
\newblock Trans. Amer. Math. Soc., 144 (8) (2016), 3283–3290.

\end{thebibliography}
\end{document}